\documentclass[a4paper,oneside,reqno,12pt]{amsart}
\usepackage[latin1]{inputenc}
\usepackage[english]{babel}
\usepackage[margin=3.25cm]{geometry}
\usepackage{enumerate}
\usepackage{amsmath}
\usepackage{amsthm}
\usepackage{color}
\usepackage{graphicx}
\usepackage{amsmath}
\usepackage{amssymb}
\usepackage{enumerate}
\usepackage{graphics}
\usepackage{hyperref}
\newtheorem{thm}{Theorem}[section]
\newtheorem{lem}{Lemma}[section]

\newtheorem{definition}{Definition}[section]
\newtheorem{proposition}{Proposition}[section]
\newtheorem{remark}{Remark}[section]

\newcommand{\R}{\mathbb{R}}

\newcommand{\ep}{\varepsilon}

\def\p{\paragraph}

\def\ds{\displaystyle}
\def\ov{\overline}
\def\O{\Omega}

\def\w{\rightharpoonup}

\def\p{\partial}
\def\ep{\varepsilon}
\def\ve{\varepsilon}

\def\w{\rightharpoonup}

\def\CD{\mathcal{D}}

\title[Homogenization-Unfolding- Monotone problem]{Homogenization of a nonlinear monotone problem in a locally periodic domain\\ via unfolding method}
\author{S. Aiyappan}
\address{Fraunhofer Institute for Industrial Mathematics ITWM, Germany}
\email{srinivasan.aiyappan@itwm.fraunhofer.de}
\author{G. Cardone}
\address{Universit\`{a} del Sannio, Dipartimento di Ingegneria, Corso Garibaldi, 107, 82100 Benevento, Italy}
\email{gcardone@unisannio.it}
\author{C. Perugia}
\address{Universit\`{a} del Sannio, Dipartimento di Scienze e Tecnologie, Via F. De Sanctis, Palazzo ex-Enel, 82100 Benevento,
Italy}
\email{cperugia@unisannio.it}
\author{R. Prakash}
\address{Department of Mathematics, Faculty of Physical Sciences and Mathematics, University of Concepcion, Chile}
\email{rprakash@udec.cl}
\date{\today}

\begin{document}

\begin{abstract}
In this paper, the asymptotic behavior of the solutions of a monotone problem posed in a locally periodic oscillating domain is studied. Nonlinear monotone boundary conditions are imposed on the oscillating part
of the boundary where as the Dirichlet condition is considered on the smooth separate part. Using the unfolding method, under natural hypothesis on the regularity of the
domain, we prove the weak $L^2$-convergence
of the zero-extended solutions  of the nonlinear problem and their flows to the solutions of a limit distributional problem.
\end{abstract}

\maketitle

\section{Introduction}\label{Intro}

This paper is concerned with the study of asymptotic analysis of a monotone problem posed on a locally periodic oscillating domain $\Omega_\ve =  \big\{ x\in \mathbb{R}^2 : 0 < x_1 < 1 ,\,\, 0 < x_2 < \eta\big(x_1, \frac{x_1}{\ve}\big) \big\},$ where $\eta$ is a locally periodic function. A non linear monotone boundary value problem has been posed on this domain $\Omega_\ve$ and its limit behavior has been analyzed when the small parameter $\ve \in \R^+$ approaches zero.  The unfolding technique has been used to analyze the limit behavior of the nonlinear problem. The novelty of this paper is understanding a nonlinear problem on a locally periodic domain whereas most of the works are linear and on periodic domains. Here the effectiveness of the unfolding technique to understand the non-linear problem will be shown while analyzing the problem. The main result consists in proving the weak $L^2$-convergence
of the zero-extended solutions  of the problem \eqref{WF} and their flows to the limit functions solving the limit problem \eqref{eq:limitproblem}, that is a distributional problem due to the presence of the function $h$ (the density of $\Omega_\varepsilon$ in $\Omega$, given in (\ref{density}) that is zero a.e.

Many physical structures can be modeled using such oscillating domains as they involve multi scales (certain parts of the structures are too small compared to the whole structure). For example heat radiators \cite{Mel-JMAA-15}, the propeller of jet engines, comb drive in micro-electro-mechanical systems, etc., Homogenization comes into play when one wants to model these structure and study their physical properties such as heat, electric, and magnetic conduction, fluid structure interaction etc. As one can see the direct numerical schemes will be impossible to implement as it involves multi scales, the homogenization helps to ease the problem.

Homogenization of boundary value problems posed on periodic oscillating rough domains has been initiated by Brizzi and Chalot with their the pioneering works in the late seventies \cite{BriCha78}. Then this direction of homogenization has attracted many mathematicians till date. There is a large literature of homogenization of such structures. Brizzi and Chalot have used extension operator to study Poisson equation on periodic oscillating domains \cite{BriCha78}. In 90's Gaudiello has studied Neumann problem with non-homogeneous boundary data \cite{Ga-RdM-94}, Kozlov,  Maz'ya, and Movchan studied such problems with the help of asymptotic expansion in the name of multi-structures \cite{KozMazMov-94}, and Nazarov analyzed in the name of singularly degenerating domains \cite{Naz-96}. Mel'nyk and his collaborators have contributed many works on this direction using asymptotic expansion method \cite{Mel1999,MelVash,Mel-JMAA-15,GaMe-JDE-18, Melnik1997,DeMaio2005,Durante2012, Melnyk2015, Gaudiello2019}. All these above works are of pillar type periodic oscillations except a few. There are some works on non uniform pillar type, that is the thickness or the cross section of the pillar changes when the height changes, see \cite{Ga-RdM-94,AiNaRa-CV,AiNaRa-CCM, Melnyk2008, Mahadevan2020} .  

The literature on locally periodic or non-periodic oscillating domains is very few. In \cite{GaGuMu-ARMA}, using Tartar's oscillating test functions method, the authors study the homogenization of Poisson problem on a non-periodic oscillating domain where the base of each uniform pillar is allowed to be non-flat. An elliptic problem with non-homogeneous non-linear boundary condition posed on a locally periodic oscillating boundary has been analyzed using a modified unfolding operator technique in \cite{AiNaRa-AMPA}. In \cite{AiKl-PP}, a locally periodic domain has been analyzed extensively with its full generality. Also see \cite{DaPe-DCDS} for locally periodic flat pillar type domains with respect to width and height of the pillars. Asymptotic analysis in thin domains with locally periodic oscillating boundary was conducted in for example  \cite{arrieta2011homogenization,ArVi-SIMA-16,ArVi-JMAA-17,chechkin1999boundary, mel2010asymptotic, pettersson2017two, Borisov2010, akimova2004, Arrieta2020}.

There are few works on non-linear problems on such domains though with more specific or more restriction on the non-linearity. In \cite{EspDonGauPic}, p-Laplacian on pillar type oscillating domains has been studied Tartar's method. Asymptotic expansion method is used in \cite{Mel-JMAA-15} to study an elliptic problem with non-linear zeroth order term and boundary data. In \cite{BlaCarGau}, a monotone problem on such periodic domain has been analyzed using Tartar's method whereas we will study the monotone problem on a locally periodic set up using unfolding technique.

Among various techniques developed to study periodic homogenization, the periodic unfolding is the recent one introduced in 2002 by  Cioranescu, Damlamian, and Griso \cite{CiDaGr-unfolding-book}, see also \cite{Cioranescu2002, CiDaGr-SIMA-08}. This method is closely related to the notion of two-scale convergence (see \cite{Ngue, allaire1992homogenization, zhikov2004two}.
Then there are different variations and modifications of the method for various problems. The method was adopted for pillar type periodic domains by Blanchard, Gaudiello, and Griso in \cite{BlGaGr-JMPA-1} and \cite{BlGaGr-JMPA-2}  and by  Damlamian and Pettersson \cite{DaPe-DCDS}. In \cite{AiNaRa-CV}, the unfolding was  modified to understand non-uniform pillar type oscillations and later for locally periodic domains \cite{AiKl-PP}. 

The rest of the article is organized as follows. In Section 2, problem description and main results are provided. Section 3 is devoted to discuss the unfolding operator and $a~priori$ estimates on the solution sequence. The proof of the main theorem is given in Section 4.  

\section{Setting of the problem and main result}\label{domain}
Let $\mathbb{T}$ denote the one-dimensional torus realized with unit measure and let $\eta : [0,1] \times \mathbb{T} \to \mathbb{R}$ be a strictly positive Lipschitz function, periodic in the second variable. Denote  any element $x\in \mathbb{R}^2$ as $x=(x_1,x_2)$ and for each $\ve = 1/k$, $k = 1, 2, \ldots$, we consider the Lipschitz domain with periodically oscillating boundary defined by
\begin{align*}
\Omega_\varepsilon & = \big\{ x\in \mathbb{R}^2 : 0 < x_1 < 1 ,\,\, 0 < x_2 < \eta\big(x_1, \frac{x_1}{\ve}\big) \big\},
\end{align*}
whose bottom boundary is given by $\Gamma_b= [0,1] \times \{0\}$.

In terms of the Lipschitz functions
\begin{align*}
\eta_-(x) & = \min_{y\in \mathbb{T}} \eta(x,y), &
\eta_+(x) & = \max_{y\in \mathbb{T}} \eta(x,y),
\end{align*}
we define our fixed domain as follows
\begin{align*}
\Omega & = \{ x \in \mathbb{R}^2 : 0 < x_1 < 1 , \,\, 0 < x_2 < \eta_+(x_1) \},
\end{align*}
which is separated into the regions
$$
\Omega^- = \{ x \in \mathbb{R}^2 : 0 < x_1 < 1 , \,\, 0 <  x_2 < \eta_-(x_1) \}
$$
and 
$$
\Omega^+ = \{ x \in \mathbb{R}^2 : 0 < x_1 < 1 , \,\,   \eta_-(x_1) <  x_2 < \eta_+(x_1) \},
$$
with interior interface $\Gamma_- = \partial \Omega^- \cap \partial \Omega^+$.

For an illustration of $\Omega_\ve$ and the corresponding $\Omega$, with the regions $\Omega^+$ and $\Omega^-$, and the interface $\Gamma_-$, see Figure~\ref{fig:domain}(a) and Figure~\ref{fig:domain}(b), respectively.

Let $k\in L^{2}(\Omega)$ and $A:\Omega_\varepsilon \times \mathbb{R}^n\to \mathbb{R}^n$ a Carath\'eodory function satisfying the following assumptions:
\begin{enumerate}\label{HpA}
\item [H1)] $A(x, \cdot) $ is strictly monotone for a. e. $x \in \Omega_\ep$,
\item [H2)] $\exists c_0>0$: $A(x, \xi) \cdot \xi \geq c_0 |\xi|^{2} - k(x), \forall \xi\in \mathbb{R}^n$,
\item [H3)] $\exists c_1>0$: $|A(x,\xi)| \leq c_1 |\xi| + k(x)$, for a. e. $x \in \Omega_\ep$ and $\forall \xi\in \mathbb{R}^n$.
\end{enumerate}
For any given function $f\in L^{2}(\Omega)$ and for any fixed $\varepsilon$, let us consider the solution $u^\varepsilon$ to the following mixed boundary value problem:
\begin{equation}\label{NHP-1}
\begin{cases}
\begin{aligned}
&-\text{div} A(x, \nabla u_{\varepsilon}) = f~\text{in}\;\O_{\varepsilon},\\
& u_\ep = 0~\text{on}~\Gamma_b, \\
& A(x, \nabla u_{\varepsilon})  \cdot \nu = 0 \;\text{on}\;\p \Omega_{\varepsilon}\setminus \Gamma_b,
\end{aligned}
  \end{cases}
\end{equation}
where $\nu$ is the unitary outward normal to $\Omega_\varepsilon$.

Let us denote by $H^1(\Omega_\varepsilon, \Gamma_b)$ the space of functions in $H^1(\Omega_\varepsilon)$ with zero trace on $\Gamma_b$. Following the monotone operator theory (see Proposition 5.1 of \cite{showa}) or following the same argument as in \cite{GaMe-JDE-18}, for any fixed $\varepsilon$, we get the existence of a unique solution $u_\varepsilon \in H^1(\Omega_\varepsilon, \Gamma_b)$ of problem \eqref{NHP-1}. \\

Moreover we can introduce the weak formulation of problem \eqref{NHP-1} as follows:
\begin{equation}\label{WF}
\left\{
\begin{array}{l}
\text{Find } u_\varepsilon \in H^1(\Omega_\varepsilon, \Gamma_b) \text{ such that}\\
\\
\displaystyle\int_{\Omega_\ep} A(x, \nabla u_{\varepsilon}) \cdot \nabla \phi ~dx= \int_{\O_\ep} f \phi ~dx,\,\forall \phi\in H^1(\Omega_\varepsilon, \Gamma_b)
\end{array}
\right.
\end{equation}
This kind of monotone problem on a fixed domain with homogeneous Dirichlet condition has been analyzed in \cite{ZhPa-2019}. 
Our goal is to describe the asymptotic behavior of the sequence of solutions $u^\ve$ as $\ve$ tends to zero and prove that it will be approximated by the solution of a problem defined in the fixed domain $\Omega$. 
\begin{figure}[!hb] 
    \centering
    \begin{minipage}{.5\textwidth}
        \centering
        \includegraphics[trim=40mm 10mm 30mm 5mm, clip,height=8cm]{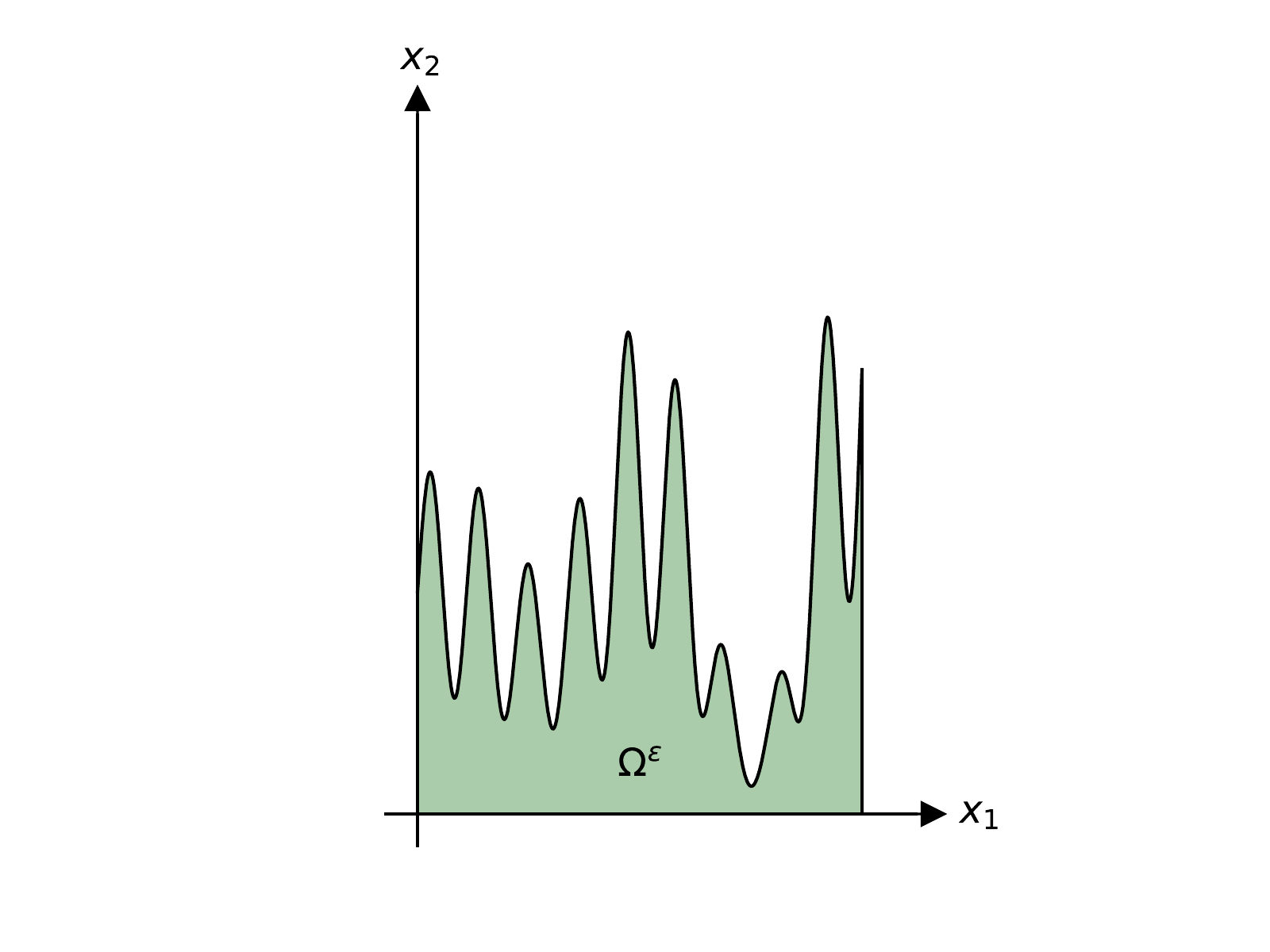}\\
        (a)
    \end{minipage}%
    \begin{minipage}{0.5\textwidth}
        \centering
        \includegraphics[trim=40mm 10mm 30mm 5mm, clip,height=8cm]{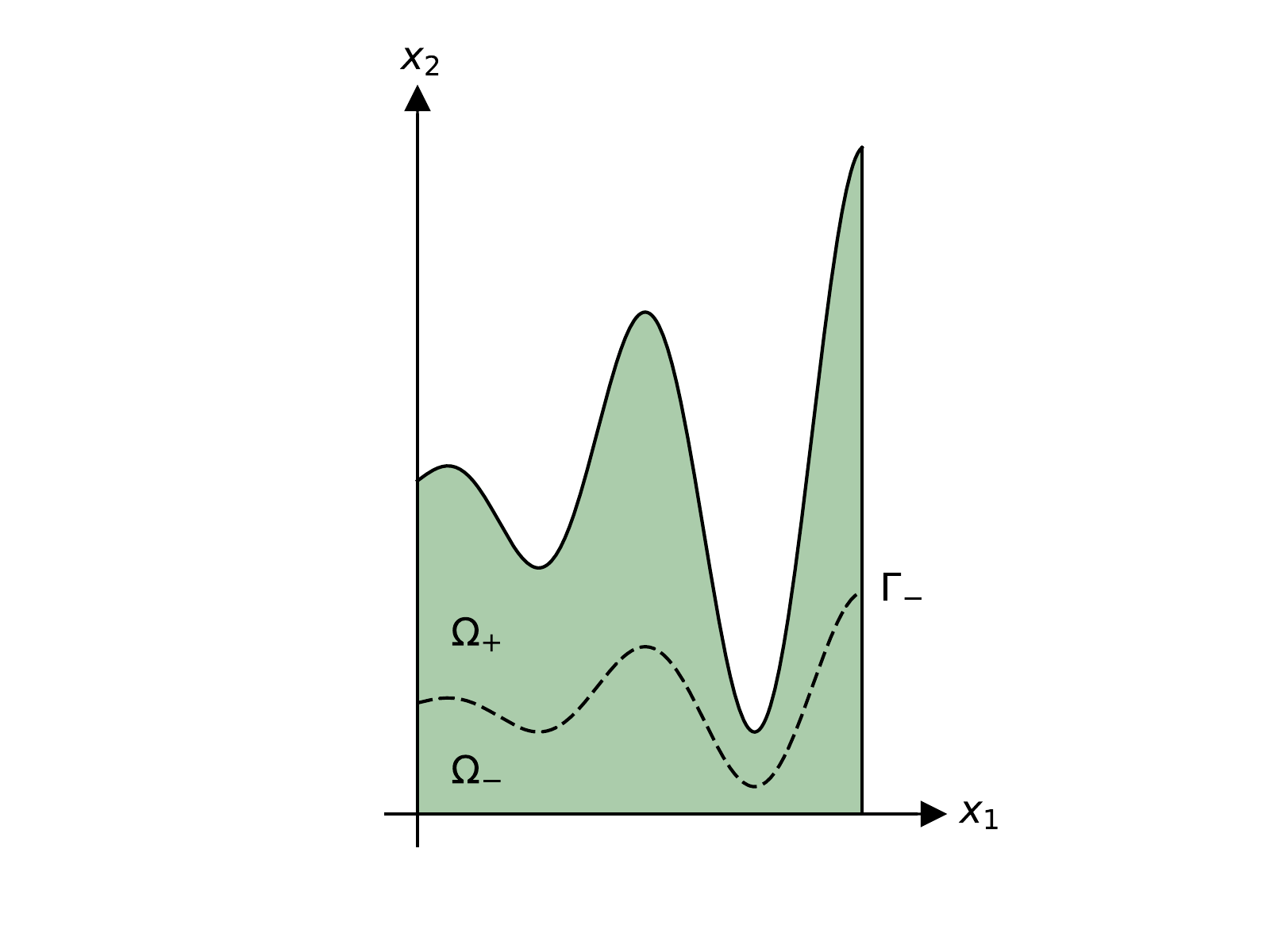}\\
        (b)
    \end{minipage}
    \caption{A locally periodic domain $\Omega^\ve$ (a), $\ve = 1/8$, and the corresponding homogeneous domain $\Omega$ (b),
                 with $\Gamma_-$ marked with a dashed line separating the regions $\Omega_+$ and $\Omega_-$.}
    \label{fig:domain}
\end{figure}
To this aim, let us observe that the assumption that $\eta$ is strictly positive ensures that the segment $\Gamma_b = [0,1] \times \{0\}$
is separated from the graph of $\eta(x_1,x_1/\ve)$, so $\Omega_-$ is a nonempty connected Lipschitz domain.
The subdomains $\Omega^+$ and $\Omega^-$ have been chosen such that $\Omega^+$ covers
the periodic region of $\Omega_\ve$, 
and $\Omega^+$ is of positive measure if $\eta(x,y)$ is non-constant in $y$ for at least one $x$.

Hence, for any $x\in \O_\ep$, we are led to consider the set 
$$Y(x)= \{ y : 0 < x_2 < \eta(x_1, y) \},$$
and to denote by

\begin{equation}\label{density}
h(x)= |Y(x)|
\end{equation}

the so called density of $\Omega_\ve$ in $\Omega$.
Let us observe that $h(x)=1$ for $x\in \Omega^-$ and $h(x)>0$ a. e. $x\in \Omega^+$ which means that the set $M=\{x\in \Omega^+:\,h(x)=0\}$ has zero measure.\\

In what follows, in order to ensure that homogenization takes place, we suppose that $Y(x)$ is connected for $x\in \Omega$ which means that 
there is only one so-called pillar or bump in each period. If we introduce the  Lebesgue space $L^2(\Omega, h)=\left\{ v : \int_{\Omega} v^2 h \,dx < \infty \right\},$ we can define the following Sobolev space
\begin{align*}
W(\Omega, \Gamma_b)
& =
\left\lbrace  v \in L^2(\Omega, h) :
     \frac{\partial v}{\partial x_1} \in L^2(\Omega^-, h)
, \,\,
     \frac{\partial v}{\partial x_2} \in L^2(\Omega, h)
, \,\,
     v = 0 \text{ on } \Gamma_b \right\rbrace.
\end{align*}
Let us observe that W is a Hilbert space with weight $h$. 
\begin{remark}
 Unlike the usual Sobolev spaces, the smooth functions need not to be dense in this weighted Sobolev space for a generic weight $h$. There are different types of necessary conditions given on $h$ by various authors for the density of smooth functions in the weighted Sobolev spaces though there is no sufficient condition.  We refer to \cite{Kufner1980,Necas2012,ChiadoSerra-94} for discussions on weighted Sobolev spaces. In this article, we do not assume any condition on $h$ except that the function $\eta$ is assumed to be Lipschitz continuous and we show the density using the unfolded domain.
\end{remark}

Throughout the paper we use the following notation:
\begin{itemize}
\item with $\widetilde{v}$ we will denote its classical extension by zero to the whole $\O$ of a function $v$ defined on $\Omega_\ep$
\item with $\ov{v}$ we will denote the following function defined in $\O^+$
$$\ov{v}(x)=\int_{Y(x)}v\,dy.$$
\end{itemize}
\medskip

We want to prove the following main result
\begin{thm}\label{tm:homogenization}
Under the assumptions $H1)\div H3)$, let $u_\ve \in H^1(\Omega^\ve, \Gamma_b)$ be the sequence of solutions to \eqref{NHP-1}. Then, there exist $u_0\in W(\Omega, \Gamma_b)$ and $\ov q \in L^2(\O^+)$ such that, as $\ve$ tends to zero, the following convergences hold
\begin{equation}\label{weak}
\begin{array}{ll}
\emph{(i)} \quad \widetilde{u^\ve} \rightharpoonup h u_0 &\quad \text{ weakly in } L^2(\Omega),\\ 
\emph{(ii)} \quad \widetilde{\nabla u^\ve} \rightharpoonup (h\ov{q},h\,\p_2 u_0) &\quad \text{ weakly in } [L^2(\Omega^+)]^2,\\
\emph{(ii)} \quad \nabla u^\ve \rightharpoonup \nabla u_0 &\quad \text{ weakly in } [L^2(\Omega^-)]^2,
\end{array}
\end{equation}
where $\sim$ denotes the classical extension to zero and the pair $(u_0, \ov q)\in W(\Omega, \Gamma_b)\times L^2(\O^+)$ is the unique solution of the following problem
\begin{equation}\label{eq:limitproblem}
\left\{
\begin{array}{ll}
\displaystyle \int_{\O^-} A(x, \nabla u_0)\nabla \phi  ~dx + 
\ds \int_{\O^+} h(x)A_2(x, \ov{q}, \p_2 u_0) \p_2\phi  ~dx = \int_{\O^-} f \phi ~dx + \int_{\O^+}f \phi h~dx\\
\\
h(x)A_1(x,\ov{q},\p_2 u_0)=0~~\text{a.e on } \O^+.
\end{array}
\right.
\end{equation}
\end{thm}

\medskip

\begin{remark}
We observe that we were not able to prove that the error
of the zeroth approximation of the solutions and their 
flows converge strongly to
zero in $L^2$ restricted to the oscillating domain as in Theorem 7.1 of \cite{AiKl-PP} because of the presence of the monotone operator $A$.
\end{remark}

\section{Preliminaries}
\subsection{The periodic unfolding operator}\label{unfolding}
The only apparent possible cause of oscillations in the solutions to (1) and their flows is the periodicity in the domain, in the $x_1$ direction. For the study of these oscillations, we will use the periodic unfolding method. In this section, we recall the definition and some properties  of the periodic unfolding operator for domains with highly oscillating smooth boundary, introduced for the first time in \cite{AiNaRa-CV}. To this aim let us define the fixed domain
\begin{align*}
\Omega_u^{+} & = \{ (x,y) \in \mathbb{R}^2 \times \mathbb{T} : 0 < x_1 < 1 , \,\, \eta_- (x_1 ) < x_2 < \eta(x_1, y) \} \\
& = \{ (x,y) \in \mathbb{R}^2 \times \mathbb{T} : x \in \Omega^+ , \,\, y \in Y(x) \}.
\end{align*}
\begin{definition}\label{defunf}
Let $v$ be a Lebesgue-measurable function defined in $\Omega^+_\ve$. The periodic unfolding operator $T_\ve$, acting on $v$, is defined as the following function in $\Omega_u^{+}$
\begin{align*}
(T^\ve v)(x,y) & = v \big( \ve \big[\frac{x_1}{\ve}\big]+\ve y,x_2\big),
\end{align*}
where $[\cdot]$ denotes the integer part and where $v$ is extended by zero when necessary.
\end{definition}
Let us consider the set
\begin{align*}
\Omega_\varepsilon^+ & = \left\lbrace x \in \mathbb{R}^2 : 0 < x_1 < 1 ,\,\, \eta_-(x_1) < x_2 < \eta\left(x_1, \frac{x_1}{\ve}\right) \right\rbrace,
\end{align*}
i.e. the region of $\Omega_\ve$ where coefficients in \eqref{NHP-1} are periodic. 
Using the previous change of variables, the characteristic function of the region of $\Omega_\ve$, $\chi_{\Omega^\ve_+}$, gives $\chi_{\Omega_\ve^u}$, the characteristic function of the domain
\begin{align*}
\Omega_\ve^u & = \left\lbrace (x,y) \in \mathbb{R}^2 \times \mathbb{T} : 0 < x_1 < 1 ,\,\, \eta_-\left(\ve \left[\frac{x_1}{\ve}\right] + \ve y \right) < x_2 < \eta\left(\ve \left[\frac{x_1}{\ve}\right] + \ve y, y\right) \right\rbrace.
\end{align*}
There holds
\begin{align}\label{eq:strong2scale}
T^\ve \chi_{\Omega_\ve^+} & = \chi_{\Omega_\ve^u} \to \chi_{\Omega_u^{+}} \quad \text{ strongly in } L^p(\mathbb{R}^2 \times \mathbb{T}), \quad 1 \le p < \infty.
\end{align}

The property~\eqref{eq:strong2scale} is the strong unfolding convergence of the sequence and it is
useful in the passage from periodic domain to a fixed domain in integrals.
It expresses that $\chi_{\Omega_+^\ve}$ converges weakly in $L^p(\mathbb{R}^2)$ while not strongly,
and that the oscillation spectrum of the sequence belongs to the integers if not empty.
To obtain~\eqref{eq:strong2scale}, one uses the almost everywhere pointwise convergence of $\chi_{\Omega^\ve_u}$ to $\chi_{\Omega_u^{+}}$ and the Lebesgue dominated
convergence theorem, or views it as a consequence of Lemma~\ref{lm:unfolding} below.

The cost of replacing in integrals the $\ve$ depending unfolded domain $\Omega_\ve^u$ with
the fixed domain $\Omega_u^{+}$ is described by the following lemma (see \cite{AiKl-PP} for details).
\begin{lem}\label{lm:unfolding}
Let $\Omega$ contain $\Omega^\ve_+$.
Suppose that $\| v^\ve \|_{L^p( \Omega )} \le C$ and $p > 1$.
Then
\begin{align*}
\int_{\Omega^\ve_+} v^\ve \,dx & = \int_{\Omega_u^{+}} T^\ve v^\ve \,dxdy + O(\ve^{1 - 1/p}),
\end{align*}
as $\ve$ tends to zero.
\end{lem}

\medskip

As a consequence of the previous lemma, we can easily prove some important properties enjoyed by $T_\ve$.

\begin{proposition}\label{properties}
The unfolding operator $\mathcal{T_\ve}$ has the following properties:
\begin{itemize}
\item[i)] For any $\ep>0$, $T_\ve$ is linear. Further, for any measurable functions $u,v:\O^+_\ep\to \mathbb{R}$, it holds $$T_\ve(uv) = T_\ve(u)T_\ve(v).$$
\item[ii)] Let $u\in H^1(\Omega_\ve^+)$. Then $ \frac{\partial}{\partial x_2}T_\ve u $ and $\frac{\partial}{\partial y}T_\ve u$ belong to $L^2(\Omega_u^{+})$ and 
\begin{equation}\label{x2}
 \frac{\partial}{\partial x_2}T_\ve u =  T_\ve\Big(\frac{\partial u}{\partial x_2}\Big),
 \end{equation}
\begin{equation}
 \frac{\partial}{\partial y}T_\ve u = \ve T_\ve\Big(\frac{\partial u}{\partial x_1}\Big).
 \end{equation}\label{x1}
\item[iii)] Let $u\in L^2(\O^+)$. Then $T^\ep u \to u$ in $L^2(\O_u^{+})$. More generally, let $\{u^\ve\}$ be a sequence of functions in $L^2(\O^+)$,  such that
 $$ u_\ve \to u \quad \hbox{strongly in } L^2(\O^+).$$
 Then
 $$T_\ve (u_\ve) \to u \quad \hbox{strongly in } L^2\big(\O_u^{+}\big).$$
 \end{itemize}
\end{proposition}
\begin{proof} Properties $i)$ and $ii)$ are easy consequences of Definition \ref{defunf} and Lemma \ref{lm:unfolding}. The proof of $iii)$ is a consequence of the analogous of  \cite{AiNaRa-CV} by a simple use of triangular inequality.
\end{proof}
\subsection{Compactness results}
\begin{lem}\label{lemest}
For any fixed $\varepsilon$, let $u_\ve \in H^1(\Omega^\ve, \Gamma_b)$ be the unique solution to \eqref{NHP-1}. Under hypotheses $H1) - H3)$, we get the following uniform estimates
\begin{equation}\label{est}
\left\{
\begin{array}{ll}
i)&  \|u_\ep\|_{L^{2}(\O_\ep)} \leq C,\\
\\
ii)&  \|A(x,\nabla u_\ep)\|_{L^{2}(\O_\ep)} \leq C,
\end{array}
\right.
\end{equation}
for a positive constant $C$ independent of $\varepsilon$.
\end{lem}
\begin{proof} By choosing $u_\ep$ as test function in the weak formulation \eqref{WF}, by assumption $H2)$ and using Poincar\'e inequality, we have
\begin{align*}
	\left( \int_{\O^\ep} c_0  |\nabla u_\ep|^2 - k(x) \right)~dx &\leq C_P \|f\|_{L^{2}(\Omega_\ep)} \left(\int_{\Omega_\ep} |\nabla u_\ep|^2 dx \right)^{1/2}.
\end{align*}
That is
\begin{equation*}
\dfrac{ c_0	 \displaystyle\int_{\O_\ep} |\nabla u_\ep|^2~ dx - \int_{\O_\ep} k ~dx}{\left(\displaystyle\int_{\Omega_\ep} |\nabla u_\ep|^2 dx \right)^{1/2}} \leq C_P \|f\|_{L^{2}(\Omega)}.
\end{equation*}
This shows that 
\begin{equation}\label{estgrad}
\|\nabla u_\ep\|_{L^{2}(\O_\ep)} \leq C.
\end{equation}
for some constant $C>0$ independent of $\ep$. 
Hence, by \eqref{estgrad} and Poincar\'e inequality, we get \eqref{est}i). Moreover, by \eqref{estgrad} and assumption $H3)$, we get \eqref{est}ii).
\end{proof}
\begin{lem}\label{lemestunf}
For any fixed $\varepsilon$, let $u_\ve \in H^1(\Omega^\ve, \Gamma_b)$ be the unique solution to \eqref{NHP-1}. Under hypotheses $H1)- H3)$, we get the following uniform estimates for the unfolded sequences
\begin{equation}\label{estunf}
\left\{
\begin{array}{ll}
i)&  \|T_\ve u_\ep\|_{L^{2}(\Omega_u^{+})} \leq C,\\
\\
ii)&  \|T_\ve \nabla u_\ep\|_{L^{2}(\Omega_u^{+})} \leq C,\\
\\
iii)&  \|T_\ve(A(\cdot,\nabla u_\ep))\|_{L^{2}(\Omega_u^{+})} \leq C,
\end{array}
\right.
\end{equation}
for a positive constant $C$ independent of $\varepsilon$.
\end{lem}
\begin{proof}
By splitting the domain $\Omega_\ep$ into the periodic  part $\Omega_\ve^+$ and the fixed part $\Omega^-$ and using  Lemma \ref{lm:unfolding}, we get \eqref{estunf}i),  \eqref{estunf}ii) and \eqref{estunf}iii) as a consequence of \eqref{est}i), \eqref{estgrad} and \eqref{est}ii) respectively.
\end{proof}

\section{Proof of Theorem \ref{tm:homogenization}. }
In this section, we establish the convergence of problem \eqref{NHP-1} to the homogenized problem \eqref{eq:limitproblem} in the sense of weak convergence of the solutions and their flows. To this aim, we will use the unfolding method whose definition and properties, we recalled in the previous section. More in particular we will use lemma \ref{lm:unfolding} to pass from the domain $\Omega_\varepsilon^+$ to the fixed domain $\Omega_u^{+}$, and the weak compactness results stated in lemmas \ref{lemest} and \ref{lemestunf} to characterize the asymptotic behavior of $u_\varepsilon$.
The proof of Theorem \ref{tm:homogenization} will be developed into six steps.\\

\textbf{Step 1.} \textit{Weak convergences}\\
By weak compactness, stated in  lemmas \ref{lemest} and \ref{lemestunf}, there exist $u_0^- \in H^1(\O^-)$ having zero trace on $\Gamma_b$, $u_0^+\in L^2(\O_u^{+})$, $(hq,d)\in (L^2(\O_u^{+}))^2$, $\zeta:=(\zeta_1,\zeta_2)\in (L^2(\O_u^{+}))^2$, $\tau\in (L^2(\O^-))^2$ and a subsequence of $\ep$, still denoted by $\ep$, such that the following convergences hold
\begin{equation}\label{conv}
\left\{
\begin{array}{lll}
i)&u_\ep\w u_0^- &~\text{weakly in } H^1(\O^-,\Gamma_b)\\
ii)&T^\ep u_\ep \w u_0^+ &~\text{weakly in } L^2(\O_u^{+}),\\
iii)&T^\ep \nabla u_\ep \w (hq,d) &~\text{weakly in } L^2(\O_u^{+}),\\
iv)& A(x,\nabla u_\ep) \w \tau &~\text{weakly in } L^{2}(\O^-),\\
v)&T^\ep A(x,\nabla u_\ep) \w \zeta &~\text{weakly in } L^{2}(\O_u^{+}).
\end{array}
\right.
\end{equation}
Since the set $M=\{x\in \Omega^+:\,h(x)=0\}$ has zero measure, we have $hq\in L^2(\Omega_u^{+})$. We want to prove that in \eqref{conv}, $u_0^+$ is independent of $y$. To this aim, let us observe that from \eqref{estunf}ii) and \eqref{x1} in Proposition \ref{properties}, we get
\begin{equation}\label{dery}
\dfrac{\partial}{\partial y}(T^\ep  u_\ep) \to 0 \text{ strongly in } L^2(\O_u^{+}).
\end{equation}
Hence the definition of weak derivative implies
$$
\int_{\Omega_u^{+}}\dfrac{\partial}{\partial y}(T^\ep  u_\ep)\varphi\,dx\,dy=-\int_{\Omega_u^{+}}T^\ep  u_\ep\dfrac{\partial \varphi}{\partial y}\,dx\,dy\quad \forall \varphi \in \mathcal{D}(\Omega_u^{+})
$$
and as $\ep$ goes to zero by \eqref{conv}ii) and \eqref{dery}, we obtain
$$
\int_{\Omega_u^{+}}u_0^+\dfrac{\partial\varphi}{\partial y}\,dx\,dy=0\quad \forall \varphi \in \mathcal{D}(\Omega_u^{+}),
$$
which means $u_0^+$ is independent of $y$.\\
By \eqref{x2} in Proposition \ref{properties} and $\eqref{conv}ii)$, or taking into account the the density of $C^\infty_0$ functions in $L^2$, we easily get $d=\p_2 u_0^+$.

Following the same argument as in \cite{AiKl-PP}, denoted by 
\begin{equation}\label{u0}
u_0=u_0^-\chi_{\O^-}+u_0^+\chi_{\O^+},
\end{equation}
we get $u_0\in W(\Omega, \Gamma_b)$. Hence in what follows, where no ambiguity arises, we can use $u_0$ in place of $u_0^+$ and $u_0^-$ respectively.
\medskip

\textbf{Step 2.} We want to prove that 
\begin{equation}\label{mean}
\ov{\zeta}_1(x)=\int_{Y(x)}\zeta_1~dy=0 \text{ a. e. } x\in\O^+.
\end{equation}
To this aim, we may use oscillating test functions as in \cite{AiKl-PP}. More precisely, let us take $\varphi \in \CD(\O^+)$ and consider the function $\phi_\ep\in H^1(\O_\ep^+)$ satisfying the following convergences
\begin{equation}\label{test}
\left\{
\begin{array}{lll}
i)&T^\ep \phi_\ep \to 0& \text{strongly in } L^2(\O_u^{+})\\
\\
ii)&T^\ep \nabla \phi_\ep \to (\varphi,0)& \text{strongly in } L^2(\O_u^{+}).
\end{array}
\right.
\end{equation}
Choosing $\phi=\phi_\ep$ as test function in the variational formulation \eqref{WF} and passing to the unfolding operator, by Lemma \ref{lm:unfolding}, we get
\begin{align}\label{UnfWeakForm}
 \int_{\O_u^{+}} T^\ep A(x, \nabla u_\ep) T^\ep \nabla \phi_\ep ~dx dy =\int_{\O_u^{+}} T^\ep f~ T^\ep \phi_\ep ~dx dy + o(1).
\end{align}
By $iii)$ in Proposition \ref{properties}, $\eqref{conv}v)$, $\eqref{test}i)$ and  $\eqref{test}ii)$, the equation \eqref{UnfWeakForm} becomes     

\begin{align*}
\int_{\O_u^{+}} \zeta_1 \varphi ~dx dy = 0,\,\forall \varphi \in \CD(\O^+),
\end{align*}
which implies $\ov{\zeta}_1=0$ almost everywhere in $\O^+$. 
\medskip

\textbf{Step 3.} \textit{Monotone relation}\\
This step is devoted to prove that for every $\Psi=(\psi_1,\psi_2)\in  (L^2(\O))^2$ the following inequality holds
\begin{equation} \label{mono-inequabis}
 \int_{\O^-} (\tau-  A(x, \Psi)) \cdot (\nabla u_0 - \Psi)  ~dx 
 + \int_{\O_u^{+}} \zeta_2 (\p_2 u_0 - \psi_2) - A(x, \Psi) ( (hq, \p_2 u_0) - \Psi)  ~dxdy \geq 0,
 \end{equation}
which will enable us to identify the functions $hq$, $\tau$ and $\zeta_2$ in $\eqref{conv}iii)$, $\eqref{conv}iv)$ and $\eqref{conv}v)$ respectively, and to derive the equation satisfied by $u_0$ in $\O^+$.\\
To this aim, let us take $\phi\in C^\infty(\overline{\O}, \Gamma_b)$ as test function in  \eqref{WF}. By unfolding and Lemma \ref{lm:unfolding}, we obtain
\begin{align*}
\int_{\O^-} A(x, \nabla u_\ep) \nabla \phi ~dx + \int_{\O_u^{+}} T^\ep A(x, \nabla u_\ep) T^\ep \nabla \phi~dx dy = \int_{\O^-}  f~ \phi ~dx \notag\\+\int_{\O_u^{+}} T^\ep f~ T^\ep \phi ~dx dy + o(1).
\end{align*}
By $\eqref{conv}iv)$ and \eqref{mean}, we get

\begin{align} \label{Tau-zeta-eqn}
\int_{\O^-} \tau \cdot \nabla \phi ~dx + \int_{\O_u^{+}} \zeta_2\, \p_2 \phi ~dx dy =  \int_{\O^-}  f~ \phi ~dx + \int_{\O_u^{+}} f~ \phi ~dx dy
\end{align}
for all $\phi\in C^\infty(\overline{\O}, \Gamma_b)$. Now, let us use the monotonicity of $A$ and by assumption $H1)$ we get

\begin{align}
\int_{\O_\ep} ( A(x, \nabla u_\ep)-A(x, \Psi)) (\nabla u_\ep - \Psi) ~dx > 0,~~\forall \Psi \in (L^2(\O))^2.
\end{align}
By splitting the domain $\O_\ep$ into $\O^-$ and $\O_\ep^+$ and the by unfolding, we obtain
\begin{align*}
&\int_{\O^-} ( A(x, \nabla u_\ep)-A(x, \Psi)) (\nabla u_\ep - \Psi) ~dx \\
&+ \int_{\O_u^{+}} ( T^\ep A(x, \nabla u_\ep)- T^\ep A(x, \Psi)) (T^\ep \nabla u_\epsilon - T^\ep \Psi) ~dxdy +o(1) > 0,~~\forall \Psi \in (L^2(\O))^2.
\end{align*}
Hence
\begin{equation}\label{ineq1}
\begin{array}{l}
 \displaystyle\int_{\O^-} ( A(x, \nabla u_\ep) \nabla u_\ep -  A(x, \nabla u_\ep) \Psi -  A(x, \Psi) \nabla u_\ep + A(x, \Psi) \Psi )~dx\\
\\
\displaystyle+\int_{\O_u^{+}} \big( T^\ep A(x, \nabla u_\ep) T^\ep \nabla u_\ep - T^\ep A(x, \nabla u_\ep) T^\ep \Psi - T^\ep A(x, \Psi) T^\ep \nabla u_\ep \\
\\
\displaystyle+T^\ep A(x, \Psi) T^\ep \Psi \big)~dxdy +o(1) > 0,~~\forall \Psi \in (L^2(\O))^2.
\end{array}
\end{equation}

At first, let us identify the limit, as $\ep$ goes to zero, of the following term in \eqref{ineq1}
$$
\int_{\O^-}  A(x, \nabla u_\ep) \nabla u_\ep ~dx+ \int_{\O_u^{+}}  T^\ep A(x, \nabla u_\ep) T^\ep \nabla u_\ep ~dxdy.
$$
To this aim let us take $u_\ep$ as test function in \eqref{WF} and pass to the unfolding operator obtaining
\begin{equation}
\begin{array}{l}
\displaystyle \int_{\O^-}  A(x, \nabla u_\ep) \nabla u_\ep ~dx+ \int_{\O_u^{+}}  T^\ep A(x, \nabla u_\ep) T^\ep \nabla u_\ep ~dxdy \\
 \\
=\displaystyle \int_{\O^-}  f~ u^\ep ~dx + \int_{\O_u^{+}}  T^\ep f  T^\ep u_\ep ~dxdy +o(1).
\end{array}
\end{equation}
When $\ep$ tends to zero, by $iii)$ in Proposition \ref{properties}, $\eqref{conv}i)$ and$\eqref{conv}ii)$, we get
\begin{equation}\label{ineq2}
\begin{array}{lll}
&&\displaystyle\lim_{\ep \to 0} \left(\int_{\O^-}  A(x, \nabla u_\ep) \nabla u_\ep ~dx+ \int_{\O_u^{+}}  T^\ep A(x, \nabla u_\ep) T^\ep \nabla u_\ep ~dxdy\right) \\
\\
&=&\displaystyle\lim_{\ep \to 0} \left(\int_{\O^-}  f~ u^\ep ~dx + \int_{\O_u^{+}}  T^\ep f  T^\ep u_\ep ~dxdy+o(1)\right)\\
\\
&=& \displaystyle\int_{\O^-}  f~ u_0^- ~dx + \int_{\O_u^{+}}  f  ~u_0^+ ~dxdy.
\end{array}
\end{equation}
If we put $\phi=u_0$ in \eqref{Tau-zeta-eqn}, we get
\begin{equation}\label{ineq3}
 \int_{\O^-}  f~ u_0^- ~dx + \int_{\O_u^{+}}  f  ~u_0^+ ~dxdy=\int_{\O^-} \tau \cdot \nabla u_0^- ~dx + \int_{\O_u^{+}} \zeta_2 \p_2 u_0^+ ~dx dy.
\end{equation}
By \eqref{ineq2} and \eqref{ineq3}, we can write
\begin{equation}\label{ineq4}
\begin{array}{lll}
&&\displaystyle\lim_{\ep \to 0} \left(\int_{\O^-}  A(x, \nabla u_\ep) \nabla u_\ep ~dx+ \int_{\O_u^{+}}  T^\ep A(x, \nabla u_\ep) T^\ep \nabla u_\ep ~dxdy\right) \\
\\
&=& \displaystyle\int_{\O^-} \tau \cdot \nabla u_0^- ~dx + \int_{\O_u^{+}} \zeta_2 \p_2 u_0^+ ~dx dy.
\end{array}
\end{equation}
Passing to the limit as $\ep \to 0$ in \eqref{ineq1}, by \eqref{conv}, \eqref{mean} and \eqref{ineq4} we obtain
\begin{align*}
&\ds \int_{\O^-} ( \tau \cdot \nabla u_0 -  \tau \cdot \Psi -  A(x, \Psi) \nabla u_0 + A(x, \Psi) \Psi )~dx\\
& \ds+\int_{\O_u^{+}} ( \zeta_2 \p_2 u_0 - \zeta_2 \psi_2 - A(x, \Psi) (hq, \p_2 u_0) + A(x, \Psi) \Psi ) ~dxdy \geq 0,~~\forall \Psi \in (L^2(\O))^2, 
\end{align*}
which means \eqref{mono-inequabis} holds true.\\
\medskip

\textbf{Step 4.} \textit{Identification of $\zeta_2$ and $\tau$}\\

In this step, it is important to recall that $$\ov{q}(x):= \int_{Y(x)} q(x,y) ~dy$$
and that $u_0$ is independent of $y$.\\

Then, for any $\lambda >0$ and $\Phi=(\phi_1,\phi_2) \in (L^2(\O))^2$, let us choose in \eqref{mono-inequabis} 
$$\Psi= \chi_{\O^+} (\ov{q}, \p_2 u_0 - \lambda \phi_2) + \chi_{\O^-} (\nabla u_0 - \lambda \Phi).$$
By considering 
\begin{align*}
\int_{\O_u^{+}} A_1(x, \ov{q}, \p_2 u_0 - \lambda \phi_2)) (hq-\ov{q})  ~dydx=0,
\end{align*}
we get
\begin{align*}
\int_{\O^-} (\tau-  A(x, \nabla u_0 -\lambda \Phi)) \cdot \Phi  ~dx + 
\int_{\O_u^{+}} (\zeta_2  - A_2(x, \ov{q}, \p_2 u_0 - \lambda \phi_2)) \phi_2  ~dxdy \geq 0,
\end{align*}
for every  $\Phi \in (L^2(\O))^2$.\\
Thus, by assumption $H3)$, as $\lambda \to 0$, we obtain
\begin{eqnarray}\label{ineq5}
\ds \int_{\O^-} (\tau-  A(x, \nabla u_0)) \cdot \Phi  ~dx + 
\int_{\O_u^{+}} (\zeta_2 - A_2(x, \ov{q}, \p_2 u_0)) \phi _2 ~dxdy \geq 0,
\end{eqnarray}
for every  $\Phi \in (L^p(\O))^2$.\\
By choosing alternatively $\Phi\in\mathcal{D}(\O^+)$ and $\Phi\in\mathcal{D}(\O^-)$ in \eqref{ineq5}, we get, respectively,
\begin{align}\label{idzeta}
\ov{\zeta}_2 = \int_{Y(x)} A_2(x,\ov{q}, \p_2 u_0) ~dy = h(x)A_2(x,\ov{q}, \p_2 u_0) ~~\text{a. e. on}~\O^+
\end{align}
and 
\begin{align}\label{idtau}
\tau = A(x,\nabla u_0)~~\text{a. e. on}~\O^-.
\end{align}
\textbf{Step 5. } \textit{$u_0\in W(\O,\Gamma_b)$ solves the homogenized problem \eqref{eq:limitproblem}}\\

By \eqref{idzeta} and \eqref{idtau}, equation \eqref{Tau-zeta-eqn} becomes
\begin{equation}\label{ineq7}
\ds \int_{\O^-} A(x, \nabla u_0)\nabla \phi  ~dx + 
\int_{\O^+} h  A_2(x, \ov{q}, \p_2 u_0) \p_2\phi  ~dx = \int_{\O^-} f \phi ~dx + \int_{\O^+}h\,f \phi ~dx.
\end{equation}
for all $\phi\in C^\infty(\overline{\O}, \Gamma_b)$.\\

Let
\begin{align*}
\Omega_{u} & = \{ (x,y) : x \in \Omega, \, y \in Y(x) \},
\end{align*}
and
\begin{align*}
W(\Omega_{u}, \Gamma_b \times \mathbb{T}) = \{ v : \, & v \in L^2(\Omega_{u}), \, \frac{\partial v}{\partial x_1} \in L^2(\Omega_- \times \mathbb{T}), \, \frac{\partial v}{\partial x_2} \in L^2(\Omega_{u}), \\
& \nabla_y v = 0 \text{ in } \Omega_{u},\, v = 0 \text{ on } \Gamma_b \times \mathbb{T} \}.
\end{align*}
Now, the equation \eqref{ineq7} can be written as 
\begin{align}\label{eq:limitW}
\int_{\Omega_{u}} \Big(  \chi_{\Omega_- \times \mathbb{T}}A(x, \nabla u_0)\nabla \phi +  \chi_{\Omega_u^{+}} A_2(x, \ov{q}, \p_2 u_0) \p_2\phi  \Big) \,dxdy & = \int_{\Omega_{u}} f\phi \,dxdy,
\end{align}
By the density of $C^\infty(\overline{\Omega},\Gamma_b)$
in $W(\Omega_{u}, \Gamma_b \times \mathbb{T})$ (as the functions in $W(\Omega_{u}, \Gamma_b \times \mathbb{T})$ are independent of $y$),
\eqref{eq:limitW} holds for any test function in $W(\Omega_{u}, \Gamma_b \times \mathbb{T})$. Hence 
the equation~\eqref{ineq7} holds for any $\phi \in W(\Omega, \Gamma_b)$.

Again, for any $\lambda >0$ and $\phi\in L^2(\O^+)$, let us choose in \eqref{mono-inequabis}
$$\Psi= \chi_{\O^+} (\ov{q}-\lambda \phi, \p_2 u_0) + \chi_{\O^-} \nabla u_0.$$
Hence, we get
\begin{eqnarray*}
\int_{\O^+}  h(x)A_1(x,\ov{q}-\lambda \phi , \p_2 u_0)\phi  ~dxdy \geq 0,\, \forall \phi\in L^2(\O^+).
\end{eqnarray*}
Thus, as $\lambda \to 0^+$ by assumption $H3)$ it holds
\begin{eqnarray*}
\int_{\O^+} h(x)A_1(x, \ov{q}, \p_2 u_0) \phi  ~dxdy \geq 0,\,\forall \phi\in L^2(\O^+)
\end{eqnarray*}
which implies
\begin{equation}\label{idA1}
h(x)A_1(x,\ov{q},\p_2 u_0)=0  ~~\text{a. e. on}~\O^+.
\end{equation}

Finally, by putting \eqref{idzeta}, \eqref{idtau} and \eqref{idA1} in \eqref{Tau-zeta-eqn}, we get that the couple $(u_0, \ov{q}) \in W(\O,\Gamma_b) \times L^2(\O^+)$ as a solution of the following problem
\begin{eqnarray}\label{lim-weak-2}
\begin{cases}
&\ds \int_{\O^-} A(x, \nabla u_0)\nabla \phi  ~dx + 
\int_{\O^+} hA_2(x, \ov{q}, \p_2 u_0) \p_2\phi  ~dx = \int_{\O^-} f \phi ~dx + \int_{\O^+}f \phi h~dx\\
&h(x)A_1(x,\ov{q},\p_2 u_0)=0~~\text{a.e on } \O^+.
\end{cases}
\end{eqnarray}
Let us observe we cannot explicitly write the previous problem as a partial differential system of equation since when we try to retrieve the boundary data on the top of the boundary by choosing a test function $\phi \in C^\infty(\Omega^+)$, we can not get any information as $h=0$ there.\\

Now, we will show that $(u_0, \ov{q}) \in W(\O,\Gamma_b) \times L^2(\O^+)$ is the unique solution of \eqref{lim-weak-2}. To this aim, let $(u_1,q_1) \in W(\O,\Gamma_b) \times L^2(\O^+)$ be another solution of \eqref{lim-weak-2}. Then,
\begin{eqnarray*}
&&\ds \int_{\O^-} (A(x, \nabla u_0) - A(x, \nabla u_1) ) (\nabla u_0 - \nabla u_1)   ~dx \\&&~~~+ 
\int_{\O^+} h(x)(A_2(x, \ov{q}, \p_2 u_0) - A_2(x, q_1, \p_2 u_1) ) (\p_2 u_0 - \p_2 u_1)  ~dx \\&&~~~~~~~~~~~~~~~= 0\\
&& h(x)(A_1(x,\ov{q},\p_2 u_0) - A_1(x,q_1,\p_2 u_1))(\ov{q}-q_1)=0~~\text{a.e on } \O^+.
\end{eqnarray*}
This implies
\begin{eqnarray*}
&\ds \int_{\O^-} (A(x, \nabla u_0) - A(x, \nabla u_1) ) (\nabla u_0 - \nabla u_1)   ~dx \\ & \ds + 
\int_{\O^+} (A(x,\ov{q}, \p_2 u_0) - A(x,q_1, \p_2 u_1) ) ( (\ov{q},h(x)\p_2 u_0) - (q_1,h(x)\p_2 u_1))  ~dx =0.
\end{eqnarray*}
As $A$ is strictly monotone, we have $\nabla u_0 = \nabla u_1 $ in $\O^-$ and 
$ (\ov{q},h(x)\p_2 u_0) = (q_1,h(x)\p_2 u_1)) $ in $\O^+$. Now, using Poincare inequality, one can easily show that $u_0=u_1$ in $\O$ in the sense of being elements of $W(\Omega, \Gamma_b)$ and $\ov{q}=q_1$ in $\O^+$.\\

The uniqueness of the solution $u_0$ to the homogenized problem  \eqref{lim-weak-2} ensures that
the full sequences in \eqref{conv} converge.

\medskip

\textbf{Step 6. } \textit{Weak limits}\\
Weak convergences $\eqref{weak}i)$ and $\eqref{weak}ii)$, as $\ep$ goes to zero,  follow from the weak unfolding limits $\eqref{conv}ii)$ and $\eqref{conv}iii)$ respectively and by taking the average over the cell of periodicity. More in particular, since $u_0^+$ doesn't depends on $y$, we get respectively
\begin{equation}\label{w1}
\widetilde{u}_\ep \w \int_{Y(x)} u_0^+\,dy=hu_0^+\quad \text{weakly in }L^2(\O^+).
\end{equation}
and 
\begin{equation}\label{w2}
\widetilde{\nabla u}_\ep \w \int_{Y(x)} \left(hq,\frac{\partial u_0^+}{\partial{x_2}}\right)dy=\left(h\ov q, h\frac{\partial u_0^+}{\partial x_2}\right) \quad \text{weakly in }L^2(\O^+).
\end{equation}
Since $h(x)=1$ in $\O^-$, \eqref{u0}, $\eqref{conv}i)$ and \eqref{w1} imply $\eqref{weak}i)$ of Theorem \ref{tm:homogenization}, while \eqref{w2} is exactly $\eqref{weak}ii)$ of Theorem \ref{tm:homogenization}.
Finally $\eqref{weak}iii)$ of Theorem \ref{tm:homogenization} is a simple consequence of $\eqref{conv}i)$.

%
%
%
%
%
%
%
%

\bibliographystyle{plain}

\bibliography{monotone}
\section*{Acknowledgement}
The fourth author would like to thank CONICYT for the financial support through FONDECYT INICIACION NO. 11180551. He would also acknowledge the support from the Facultad de Ciencias Fisicas y Matematicas, Universidad de Concepcion (Chile) as this research was initiated during the visit of the first autor there in August 2019. The second and the third authors would to show their gratitude to GNAMPA (INDAM) for all the necessary support provided.

\end{document}